\documentclass[11pt]{article}
\usepackage{geometry}                
\usepackage{setspace}

\usepackage{amsmath}
\usepackage{amsthm}
\usepackage{mathtools}
\usepackage{mathrsfs}

\usepackage{graphicx}
\usepackage{amssymb}
\usepackage{epstopdf}
\usepackage{graphicx}
\usepackage{caption}
\usepackage{subcaption}
\usepackage{epstopdf}
\usepackage{fancyref}

\usepackage[usenames,dvipsnames,x11names]{xcolor}
\usepackage{tikz}
\usepackage[all]{xy}
\usetikzlibrary{arrows,snakes,shapes}
\theoremstyle{plain}

\theoremstyle{definition} 
\newtheorem{definition}{Definition}

\newtheorem{thm}{Theorem}
\newtheorem{lemma}{Lemma}
\newtheorem{cor}{Corollary}
\newtheorem{theorem}{Theorem}

\theoremstyle{remark}
\newtheorem*{remark}{Remark}

\numberwithin{equation}{section}
\numberwithin{thm}{section}
\numberwithin{lemma}{section}
\numberwithin{cor}{section}

\title{Bounds on \"{U}bercrossing and Petal Numbers for Knots}
\author{Colin Adams,
Orsola Capovilla-Searle,
Jesse Freeman,
Daniel Irvine,\\
Samantha Petti,
Daniel Vitek,
Ashley Weber,
Sicong Zhang} 

\begin{document}

\maketitle


\abstract{ An $n$-crossing is a point in the projection of a knot where $n$ strands cross so that each strand bisects the crossing. An \"ubercrossing projection has a single $n$-crossing and a petal projection has a single $n$-crossing such that there are no loops nested within others. The \"ubercrossing number, $\text{\"u}(K)$, is the smallest $n$ for which we can represent a knot $K$ with a single $n$-crossing. The petal number is the number of loops in the minimal petal projection. In this paper, we relate the \"{u}bercrossing number and petal number to well-known invariants such as crossing number, bridge number, and unknotting number. We find that the bounds we have constructed are tight for $(r, r+1)$-torus knots. We also explore the behavior of \"{u}bercrossing number under composition.}


\section{Introduction}
Knot theorists have traditionally considered projections of knots where crossings consist of a single understrand and a single overstrand. In \cite{Triple},  $n$-crossings, also known as multi-crossings, were introduced for projections of knots (having previously been introduced for projections of graphs). An $n$-crossing is a point in a knot or link projection where $n$ strands cross so that each strand bisects the crossing.  An $n$-crossing has $n$ strands that are labelled top to bottom $1, 2, ... , n$ respectively. There are various types of $n$-crossings characterized by where the strands with different heights are located. For all $n$-crossings, we read the height of the strands clockwise around the crossings and always beginning with the top strand. In \cite{Triple}, it was proved that for every $n \ge 2$, every knot and link has a projection such that all crossings are $n$-crossings. Hence, we can define $c_n(K)$ to be the least number of $n$-crossings in such a projection for the knot or link $K$. In \cite{Multi}, the authors proved that every knot has a projection with a single multi-crossing. In other words, given a knot or link $K$, there exists a positive integer $n$ such that $c_n(K) = 1$.

\begin{definition}
An {\bf \"ubercrossing projection} of a knot $K$ is a projection of $K$ with a single multi-crossing.
The {\bf  \"ubercrossing number} of a knot $K$, $\text{\"{u}}(K)$, is the the smallest $n$ for which there exists a projection with a single $n$-crossing. 
\end{definition}
 Note that an \"ubercrossing projection may have loops contained within other loops.  We call a loop that contains at least one other loop a {\bf nesting loop}. See Figure \ref{fig:uber}.
 
\begin{figure}[h]
\begin{center}
\includegraphics[scale=0.2]{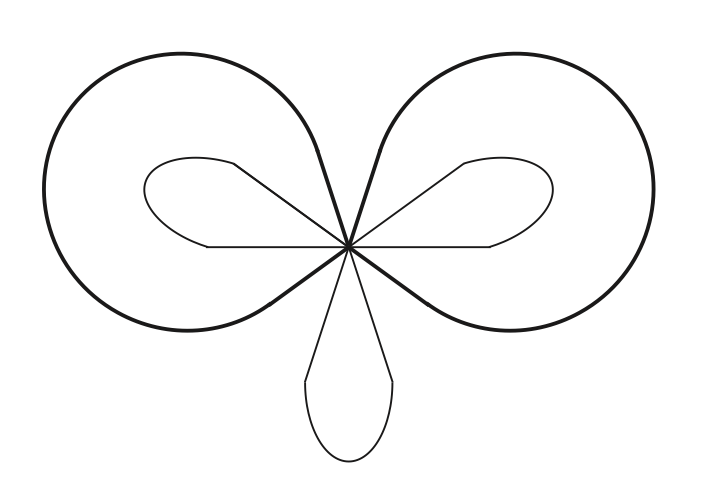}
\caption{An \"ubercrossing projection with two nesting loops.}
\label{fig:uber}
\end{center}
\end{figure}

To show that every knot has an \"ubercrossing projection, the authors of \cite{Multi} provided an algorithm, known as the petal algorithm, that takes any classical projection of a knot $K$ with only double crossings and generates an \"ubercrossing projection of $K$ with only one nesting loop that has just under half the other loops contained within it \cite{Multi}. This type of projection is known as a {\bf pre-petal projection} because we can fold the nesting loop into the central crossing and obtain a petal projection as shown in Figure \ref{fig:prepetal}.

\begin{figure}[h]
\begin{center}
\includegraphics[scale=0.3]{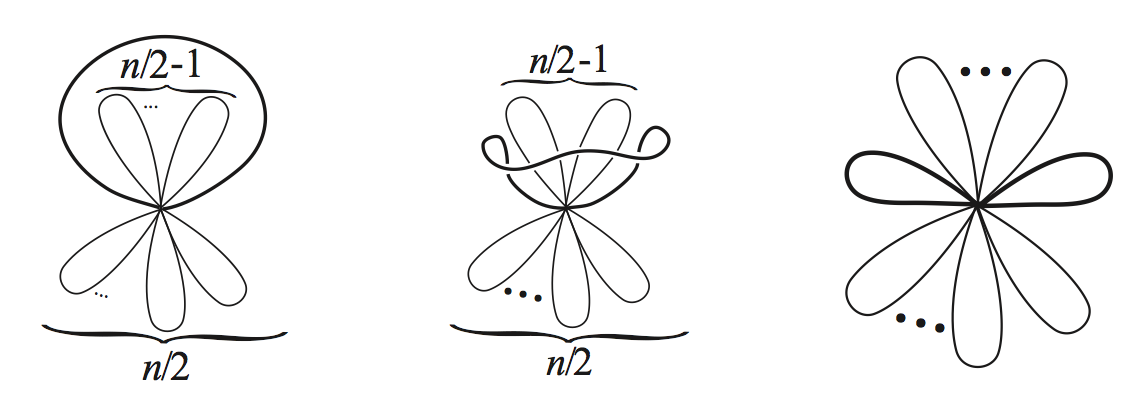}
\caption{Obtaining a petal projection from a pre-petal projection.}
\label{fig:prepetal}
\end{center}
\end{figure}

\begin{definition}
A {\bf petal projection} of a knot $K$ is a projection of $K$ with a single multi-crossing such that there are no nesting loops.
The {\bf petal number}, p(K), is the number of loops in the minimal petal projection, or equivalently, the
number of strands passing through the single $n$-crossing. A {\bf petal permutation} is the permutation of $n$ integers obtained by listing the labels corresponding to the heights of the $n$ strands as we travel clockwise around the single multi-crossing starting with the topmost strand, which is labelled 1.\end{definition}

Note that petal number is not defined for arbitrary links of more than one component. There are very few links that possess petal projections. Also note that the top strand of a petal projection can be pulled off and slid around to the bottom, resulting in a new petal permutation for the same knot. We consider this an equivalent permutation.

Let $T_{p,q}$ be the $(p,q)$-torus knot and $\overline{T_{p,q}}$ its mirror image. In \cite{Multi}, the petal number of $T_{r,r+1}$ was shown to be $2r+1$. In Section~\ref{bridge} we relate the \"ubercrossing number of a knot or link to the bridge number and show that for all $(r,r+1)$-torus knots, $\text{\"u}(T_{r,r+1})=2r$. This is the first infinite class of knots for which  $\text{\"u}(K)$ has been determined.We then investigate how the \"ubercrossing number behaves under composition and show that for any two knots $K_1$ and $K_2$, $\ddot{u}(K_1\#K_2)\leq \min\{\ddot{u}(K_1)+p(K_2)-1,\ddot{u}(K_2)+p(K_1)-1\}$. We use this to determine the exact \"ubercrossing number of all compositions of various torus knots of the form $T_{r,r+1}$. 

In Section \ref{unknotting} we find the following upper bound on the unknotting number, $u(K)$, in terms of the petal number, $u(K) \leq \frac{(p(K)-1)(p(K)-3)}{8}$. We show that this bound is realized exactly for only the $(r,r+1)$-torus knots. This implies that the  minimal petal permutations that represent the $(r,r+1)$-torus knots are unique up to the aforementioned equivalence.

In Section  \ref{petal into star} we relate petal number to the traditional crossing number $c(K)$ and find that $c(K)\leq\frac{p^2-2p-3}{4}$. Furthermore, the bound is realized for $(r,r+1)$-torus knots. 

In Section ~\ref{algorithm} we explore the petal algorithm and find that the minimal petal projection of any knot $K$ can be obtained by applying the petal algorithm to some projection of the knot. However, we show that the projection to which the algorithm is applied need not be a minimal crossing projection, by demonstrating that performing the petal algorithm on the minimal crossing projections of any two-braid knot does not generate a minimal petal projection.


\section{Bounds on \"{U}bercrossing Number} \label{bridge}

In this section, we first present a lower bound on \"{u}bercrossing number in terms of bridge number and show that this lower bound is realized for $(r,r+1)$-torus knots.

\begin{lemma}\label{bridge} $\ddot{u}(K) \ge 2b(K)$, where $b(K)$ is the bridge number.
\end{lemma}
\begin{proof}
It suffices to find a projection of $K$ with at most $\ddot{u}(K)$ local extrema. First, embed the knot in space so that when we look down an axis we see an \"ubercrossing projection realizing the \"ubercrossing number. Let $A$ be the axis that passes through the \"ubercrossing and is perpendicular to the projection plane. Now consider the projection of the knot onto a plane so that the $A$ axis appears vertical, as in Figure~\ref{fig:uber-rotate}. 

\begin{figure}[h]
\begin{center}
\includegraphics[scale=0.3]{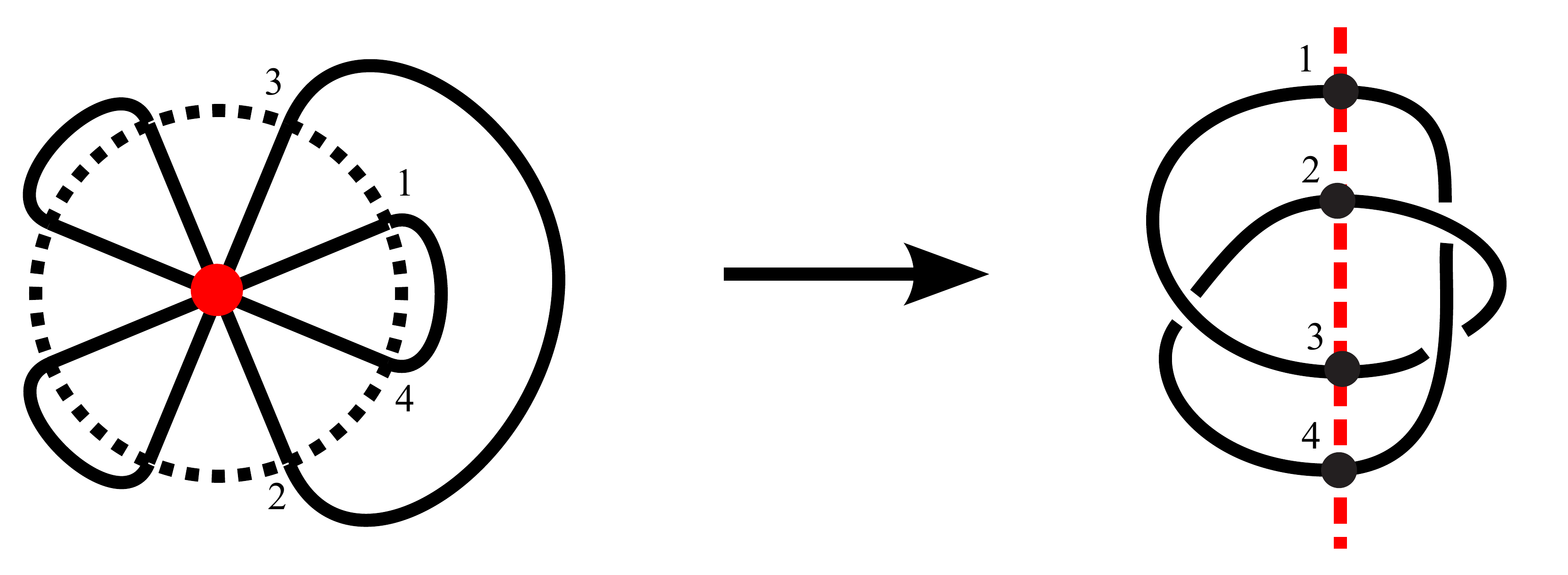}
\caption{A minimal \"{u}bercrossing projection of the trefoil knot isotoped so that $A$ appears vertical.}
\label{fig:uber-rotate}
\end{center}
\end{figure}

Because each strand in the single $n$-crossing of the \"{u}bercrossing projection passes through the axis $A$, the new projection of $K$ will cross $A$ in at least $\ddot{u}(K)$ points. However, we can isotope the knot so that the arcs that connect these points on $A$ are monotonically decreasing from their high point on $A$ to their low point on $A$. In this projection, the aforementioned $\ddot{u}(K)$ intersections of $K$ with $A$ are the only possible local extrema.  

\end{proof}

\begin{cor}
 $\ddot{u}(T_{r,r+1}) = 2r$ for all $r\ge2$.
\end{cor}
\begin{proof}
 By \cite{Schubert}, $b(T_{p,q}) = \min(p,q)$. In \cite{Multi}, it was proved that $\ddot{u}(K) \le p(K) - 1$ for any link $K$ and $p(T_{r,r+1}) = 2r+1$. Combining this upper bound with the lower bound in Lemma \ref{bridge}, we obtain$$2r = 2b(T_{r,r+1}) \le \ddot{u}(T_{r,r+1}) \le p(T_{r,r+1}) - 1 = 2r.$$ Hence $\ddot{u}(T_{r,r+1}) = 2r$ for all $r \ge 2$. \end{proof} 

We now turn to composition. It has been conjectured for over one hundred years that $c(K_1\#K_2)=c(K_1)+c(K_2)$ and for many knots this is known to be true, including alternating knots, torus knots and compositions of these two types. However, in \cite{Adams1},  examples are given of knots for which the $n$-crossing number $c_n(K)$ is known to be sub-additive under composition for $n \geq 4$. We  continue this line of inquiry by considering how the \"ubercrossing number behaves under composition. 

\begin{definition} A {\bf ribbon} consists of two strands passing through the single multi-crossing in  an \"ubercrossing projection that are connected to one another, forming a loop. In a {\bf left ribbon} the overstrand of the ribbon forms the left half of the loop. In a {\bf right ribbon} the overstrand of the ribbon forms the right half of the loop. 
\end{definition}

\begin{figure}[h]
\begin{center}
\includegraphics[scale=.2]{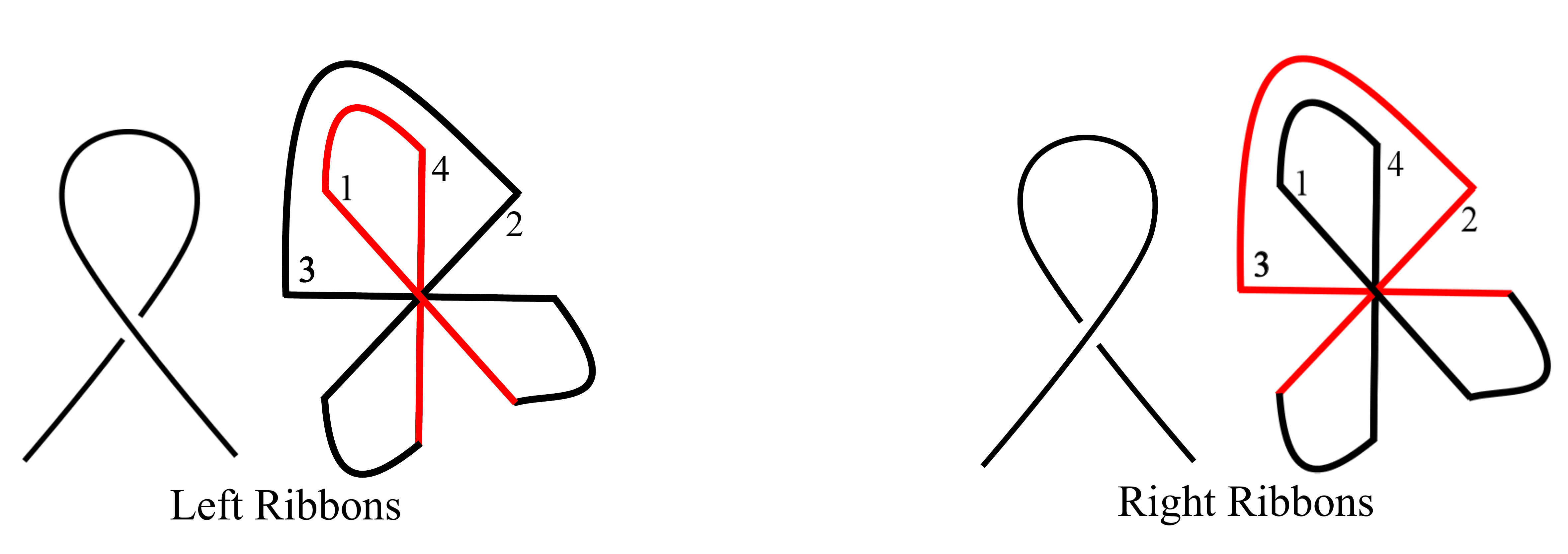}
\caption{Pictured are two of the ribbons in the trefoil \"ubercrossing projection, which are highlighted.} 
\end{center}
\end{figure}

\begin{definition} 
In a minimal pre-petal projection, $K$, the {\bf bottom ribbon} is a ribbon containing the bottom strand such that its  loop is contained in the nesting loop of the pre-petal projection.  
\end{definition}

It is important to note that in a minimal pre-petal projection the bottom strand will never be part of the nesting loop. If it were then it would not be the pre-petal projection corresponding to the minimal petal projection because we could simply remove the nesting loop and obtain a petal projection with fewer petals.

\begin{theorem} \label{comp} Let $K_1$ and $K_2$ be knots, where the minimal \"{u}bercrossing projection of $K_2$ is a pre-petal projection. Without loss of generality, assume \\ $\ddot{u}(K_1)+p(K_2)\leq \ddot{u}(K_2)+p(K_1)$. If one of the following conditions holds, then
\[\ddot{u}(K_1\#K_2)\leq \ddot{u}(K_1)+p(K_2)-3.\] 
\begin{enumerate}
\item The bottom ribbon of the minimal pre-petal projection of $K_2$ is a right ribbon and there exists a left ribbon in the minimal \"ubercrossing projection of $K_1$. 
\item The bottom ribbon of the minimal pre-petal projection of $K_2$ is a left ribbon and there exists a right ribbon in the minimal \"ubercrossing projection of $K_1$. 
\end{enumerate}
\end{theorem}

\begin{proof}
Assume Condition 1, where $K_1$ in a minimal \"ubercrossing projection with a left ribbon and $K_2$ in its minimal pre-petal projection with a right bottom ribbon. The case for Condition 2 follows similary. We provide a means of composing $K_1$ and $K_2$ in order to obtain an \"ubercrossing projection with $\ddot{u}(K_1)+p(K_2)-3$ strands passing through the \"ubercrossing. We illustrate this process in Figure \ref{steps}.

\begin{figure} [h!]
	\centering
	\begin{subfigure}[b]{0.3\textwidth}
		\centering
		\includegraphics[scale=.2]{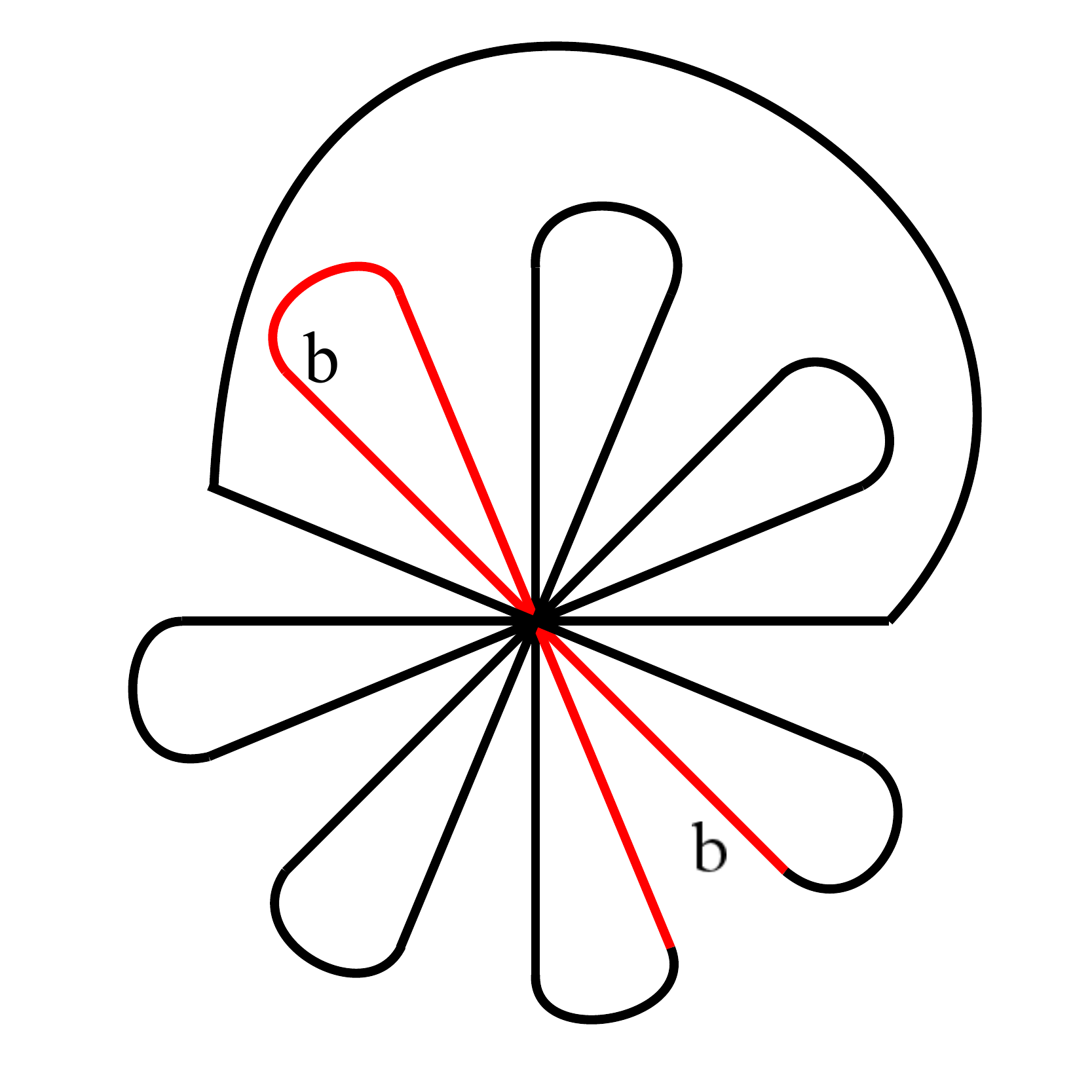}
		\caption{}
	\end{subfigure}
	\begin{subfigure}[b]{0.3\textwidth}
		\centering
		\includegraphics[scale=.2]{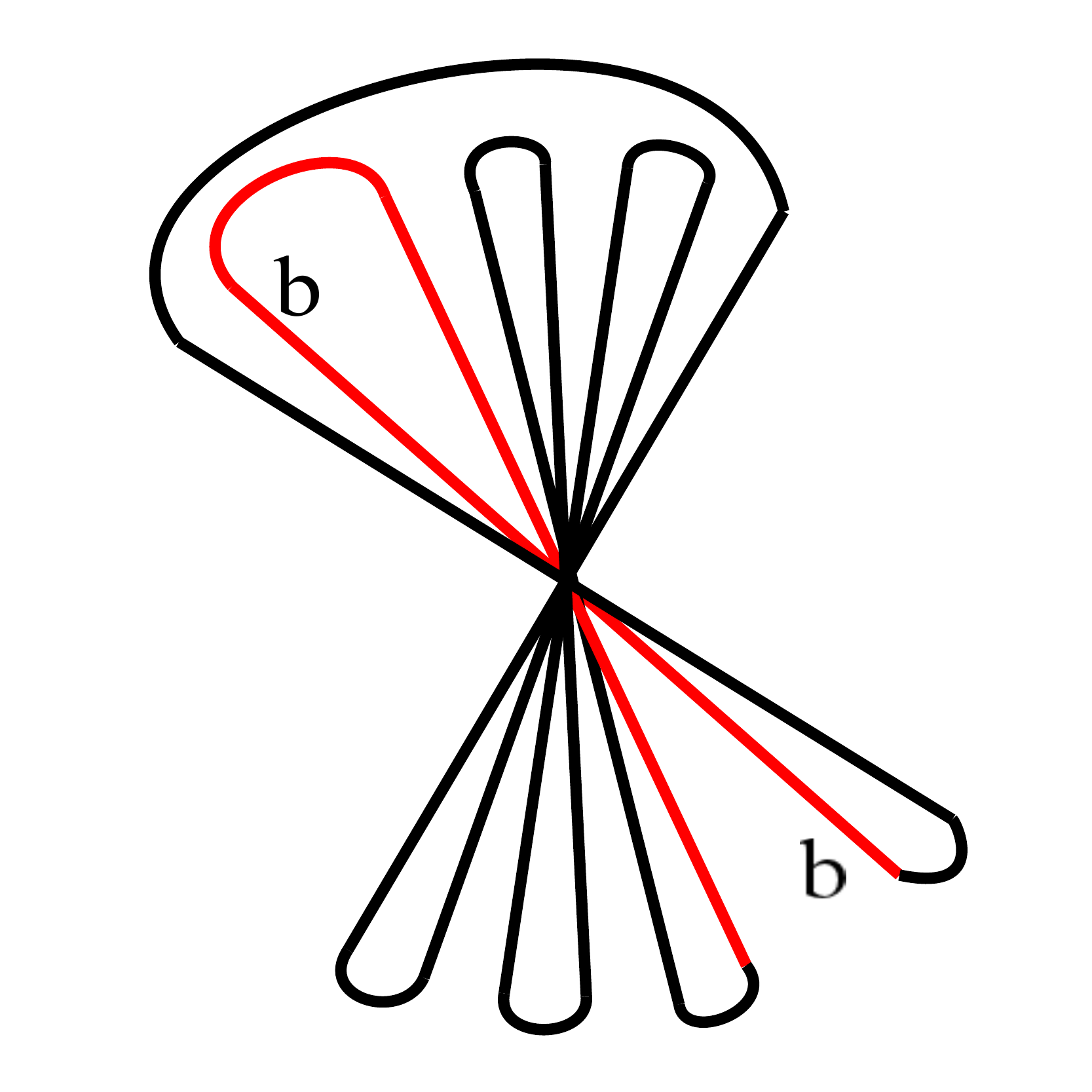}
		\caption{}
	\end{subfigure}
	\begin{subfigure}[b]{0.3\textwidth}
		\centering
		\includegraphics[scale=.2]{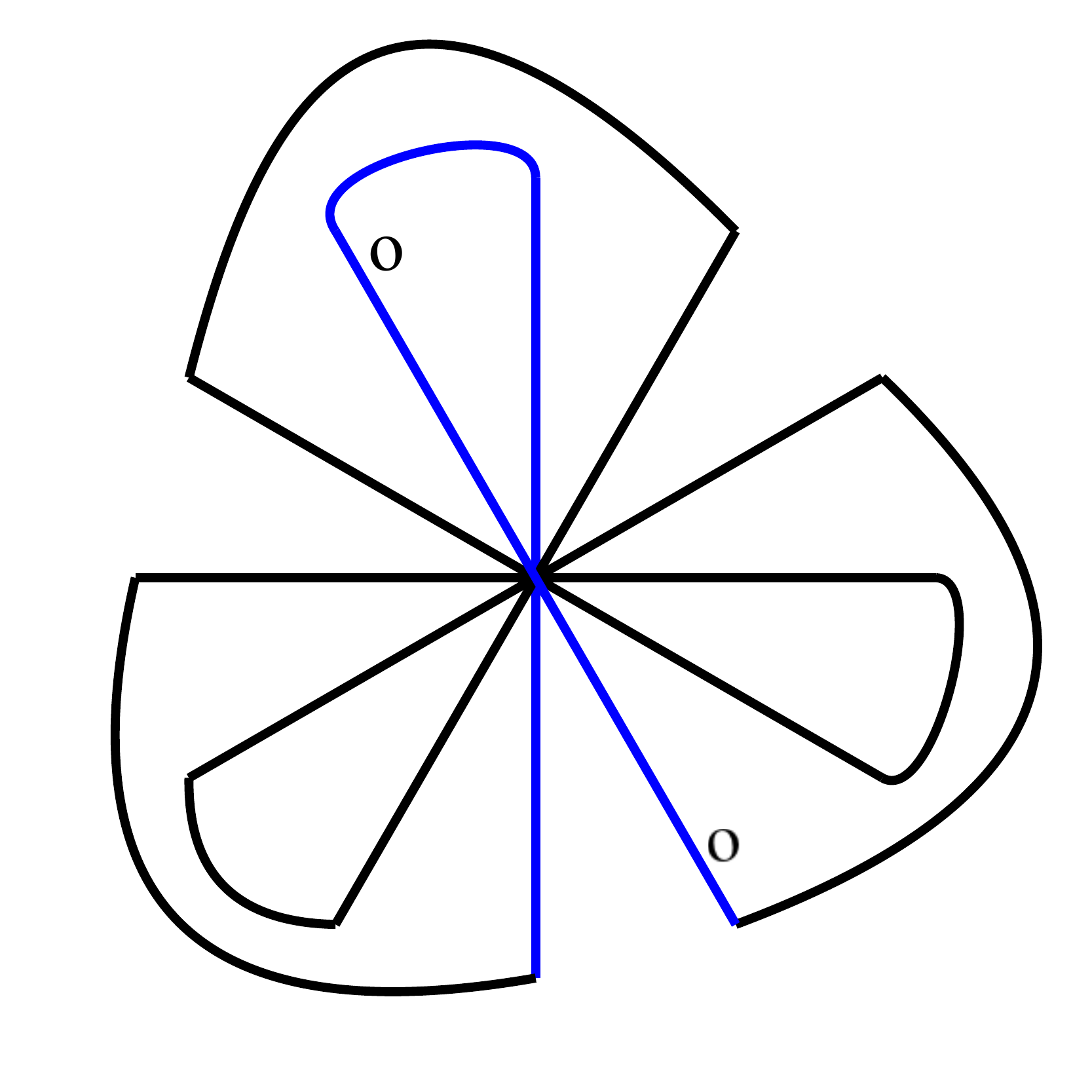}
		\caption{}
	\end{subfigure}
	\begin{subfigure}[h]{0.3\textwidth}
		\centering
		\includegraphics[scale=.2]{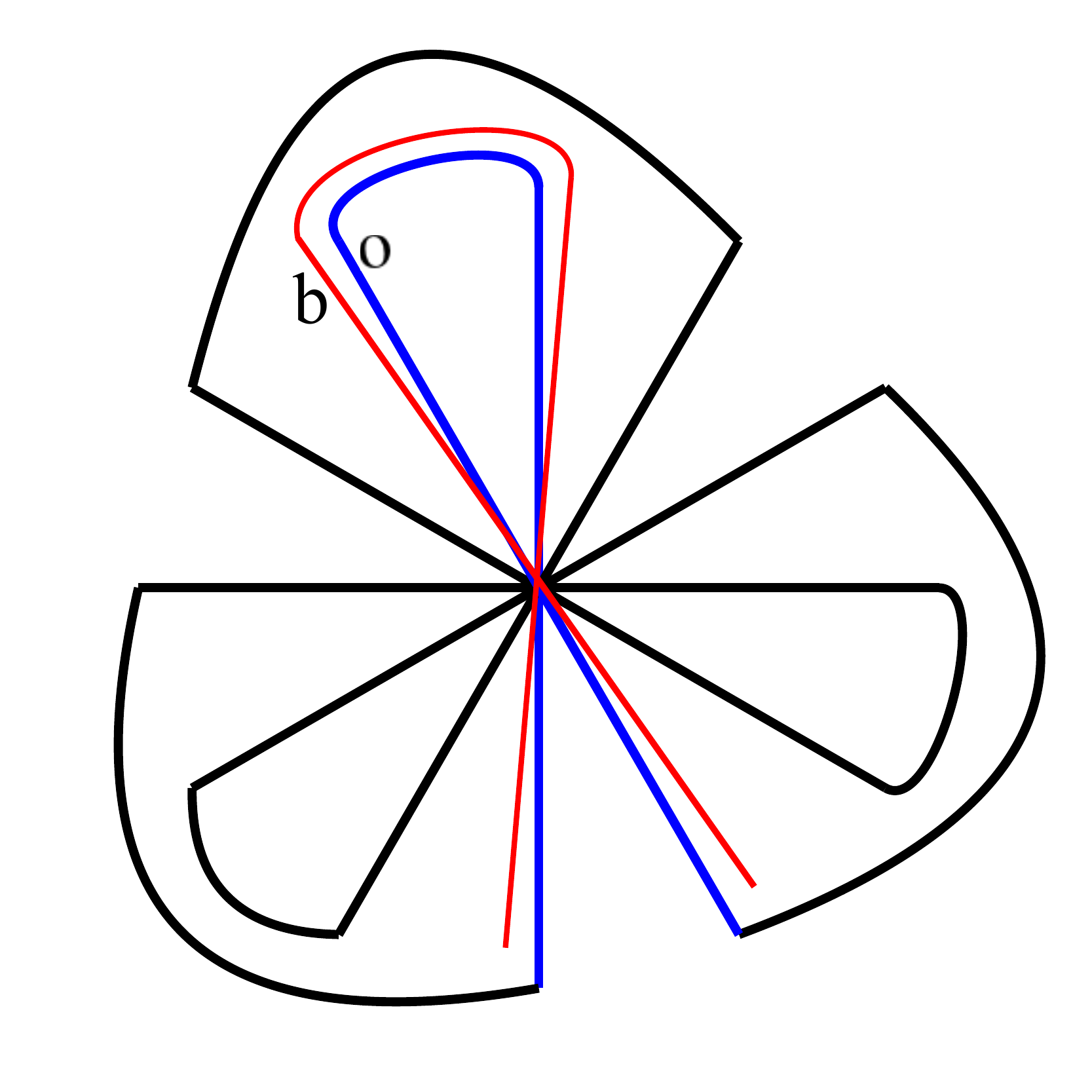}
		\caption{}
	\end{subfigure}
	\begin{subfigure}[h]{0.3\textwidth}
		\centering
		\includegraphics[scale=.2]{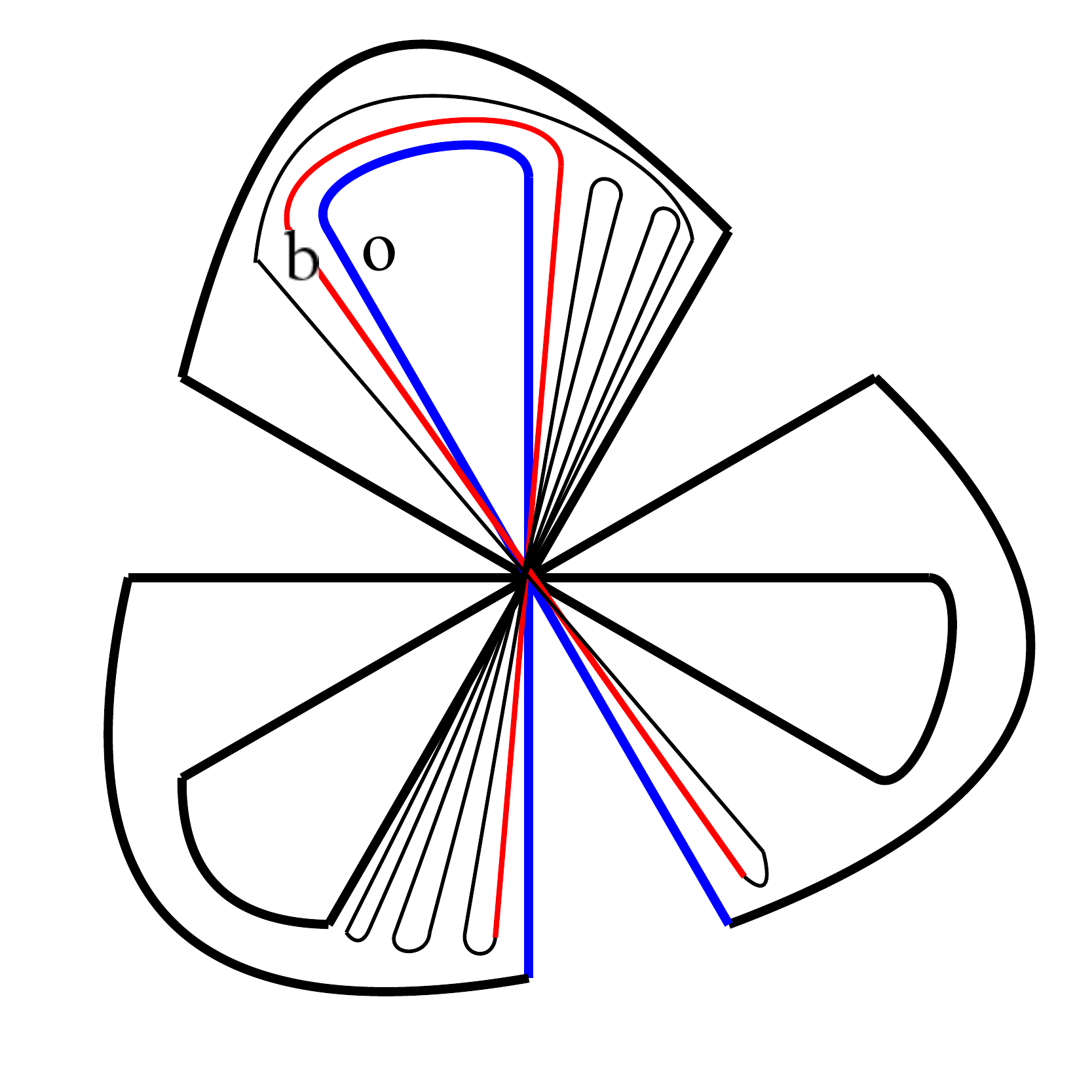}
		\caption{}
	\end{subfigure}
	\begin{subfigure}[h]{0.3\textwidth}
	\centering
		\includegraphics[scale=.2]{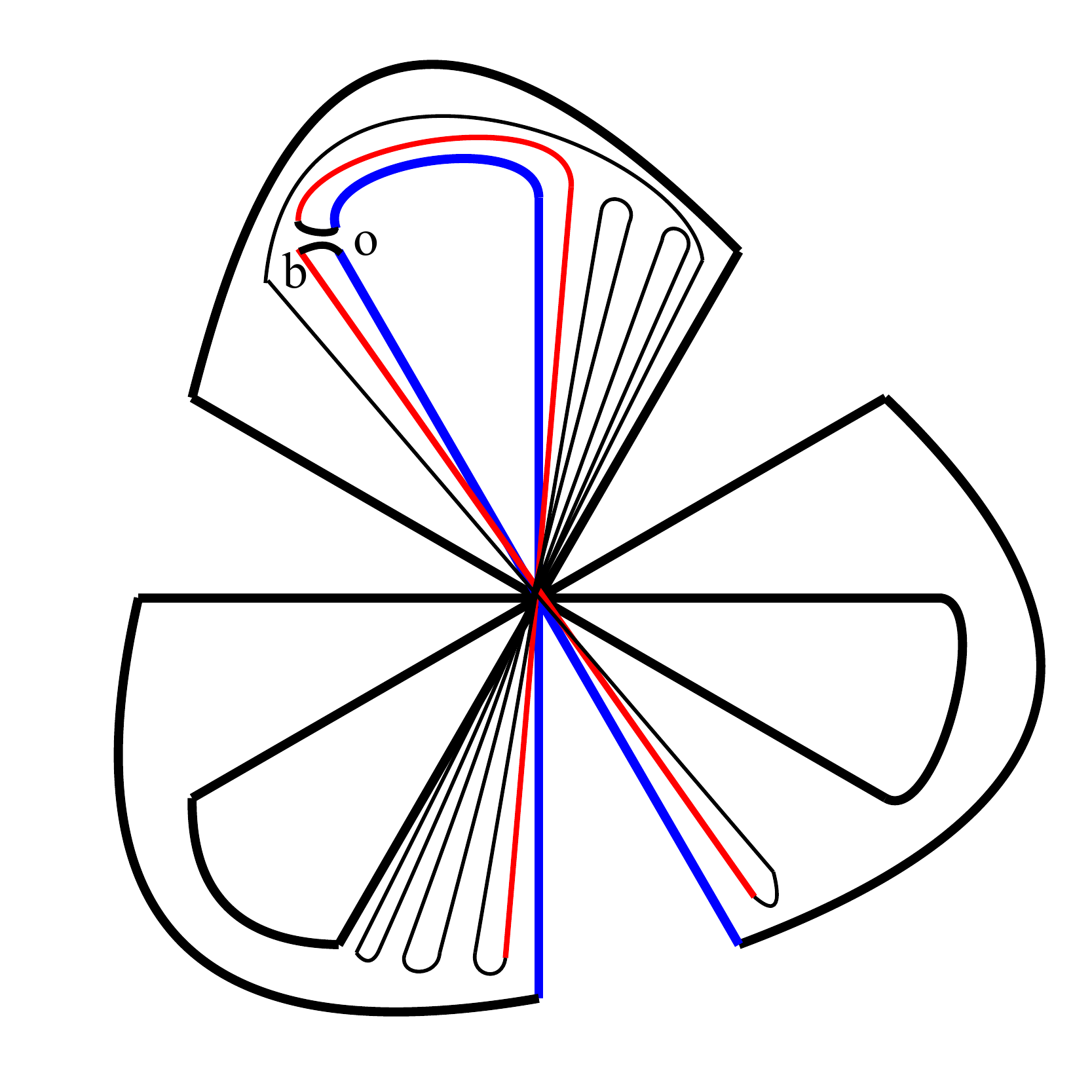}
		\caption{}
	\end{subfigure}
	\begin{subfigure}[h]{0.3\textwidth}
		\centering
		\includegraphics[scale=.2]{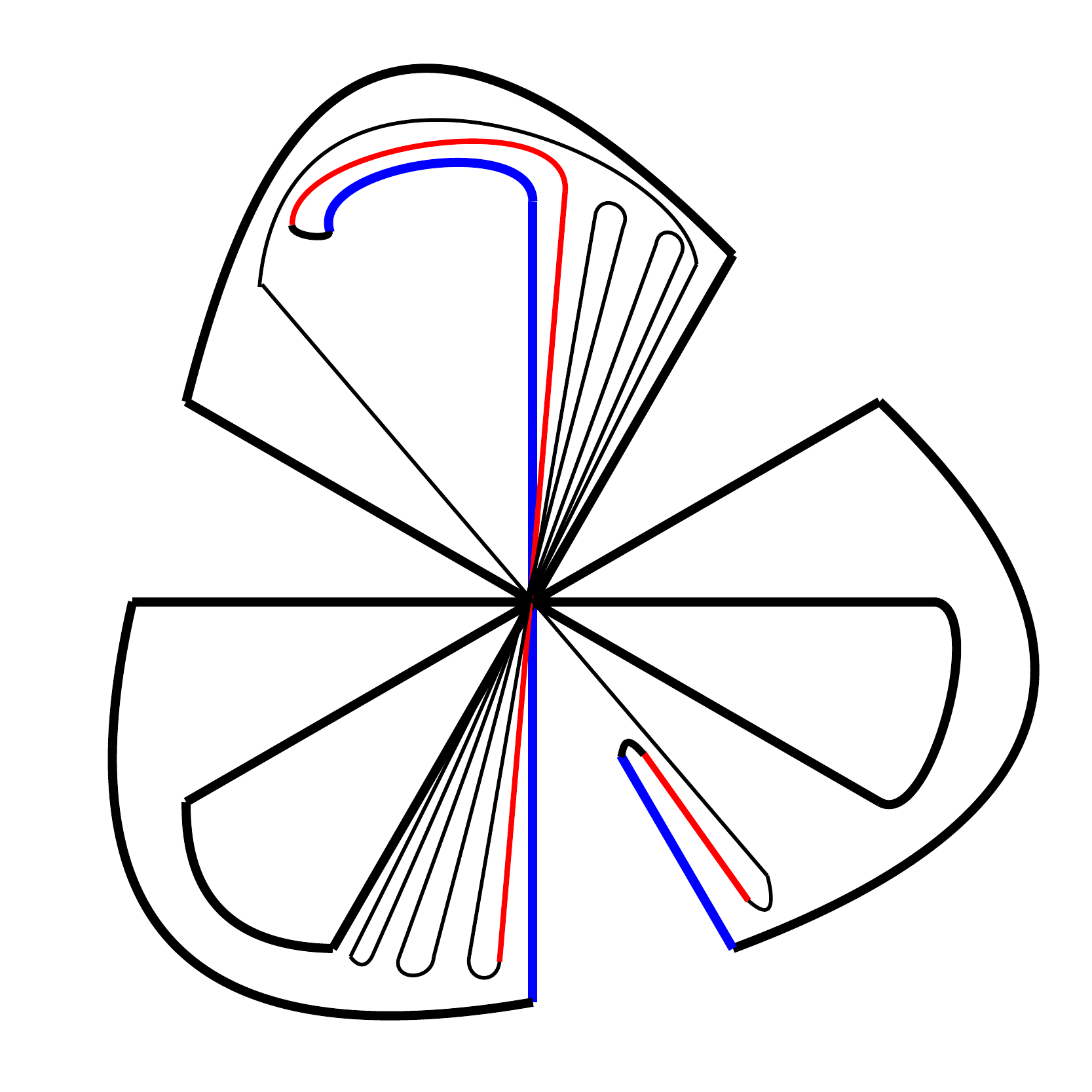}
		\caption{}
	\end{subfigure}
	\begin{subfigure}[h]{0.3\textwidth}
		\centering
		\includegraphics[scale=.2]{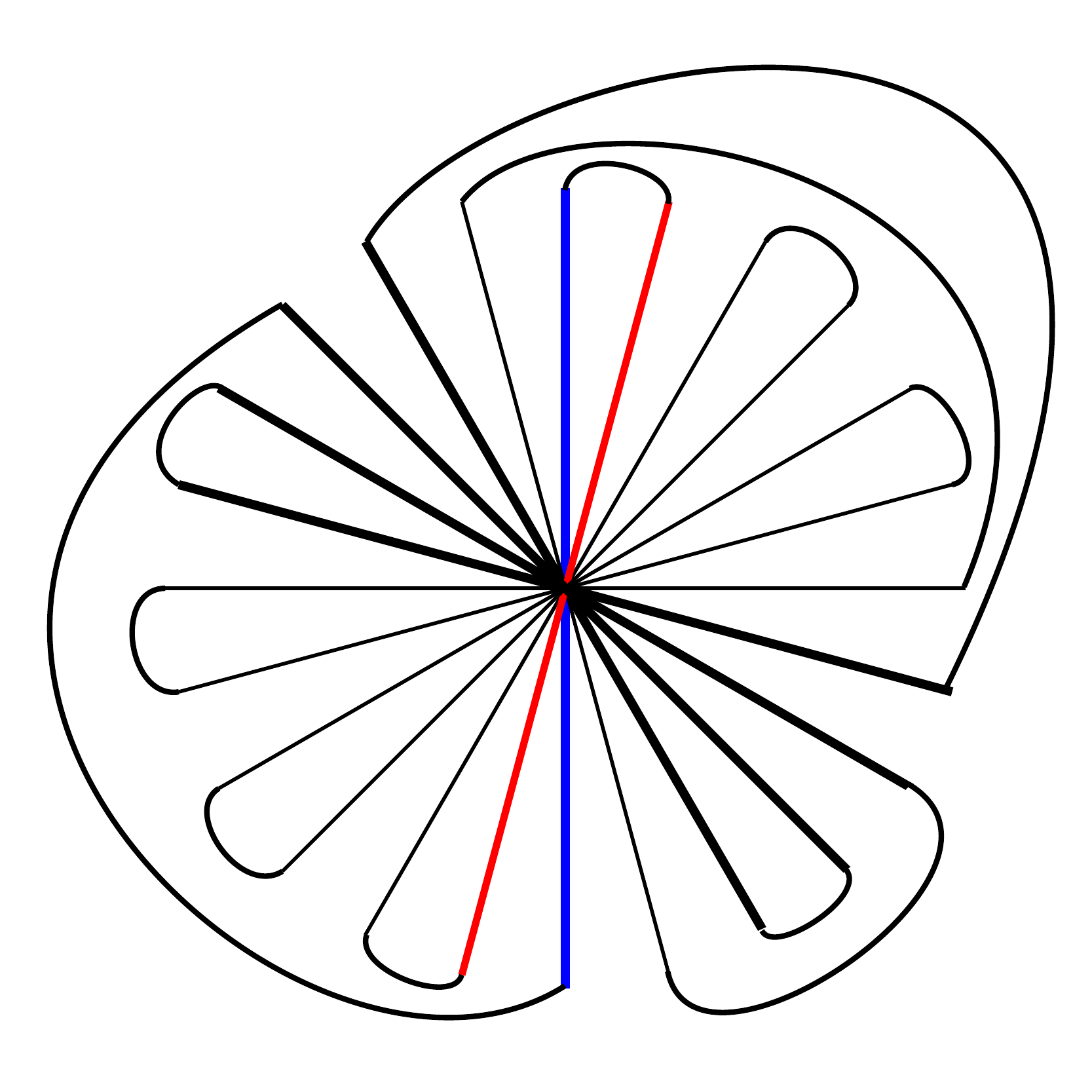}
		\caption{}
	\end{subfigure}
\caption{The steps for composing.}
\label{steps}
\end{figure}

\begin{enumerate}
\item  Identify the bottom ribbon of $K_2$ in its minimal pre-petal projection as in Figure \ref{steps}(a). Recall that we are assuming that the bottom ribbon is a right ribbon.

\item Isotope $K_2$ so that it ``hugs'' the bottom ribbon, as in Figure \ref{steps}(b). Pull the parts of the knot contained in the nesting loop in so that they closely surround the loop of the bottom ribbon. The other strands should be pulled close to ends of the bottom ribbon. 

\item  Identify a left ribbon in $K_1$ as in Figure \ref{steps}(c).  There may be more than one left ribbon in $K_1$, however it does not matter which one we choose. Label the overstrand within the left ribbon $o$. Now we are prepared to draw an \"ubercrossing projection of the the disjoint union of $K_2$ and $K_1$. We will begin with a projection of $K_1$ and add the projection of $K_2$ in the following two steps. In Step 4 we draw in the bottom ribbon of $K_2$. In Step 5 we draw the rest of $K_2$.

\item Draw the bottom ribbon of $K_2$ around the selected left ribbon in $K_1$ so that the bottom ribbons nests the left ribbon, as in Figure \ref{steps}(d).  Additionally, the strands of the bottom ribbon bisect the multi-crossing of $K_1$. The bottom strand of the the bottom ribbon of $K_2$ should be one level above the overstrand of the selected left ribbon, which we labelled $o$ in the crossing. 

\item Draw the rest of $K_2$ so that $K_2$ is contained entirely in the regions between the selected left ribbon in $K_1$ and the strands of $K_1$ adjacent to it, as in Figure \ref{steps}(e). Think of this as copying the projection in Step 2 and placing it on top of the the projection in Step 4 so that the bottom ribbon is aligned. Each strand of $K_2$ will bisect the crossing. The result is a projection of a disjoint union of $K_1$ and $K_2$ in which $\ddot{u}(K_1)+p(K_2)-1$ strands bisect the single multi-crossing in the projection.

\item Compose $K_1$ and $K_2$ in the following manner as in Figure \ref{steps}(f).  Cut the bottom strand of $K_2$, which we call strand $b$, and the overstrand of the selected left ribbon of $K_1$, which we call strand $o$. Glue together the cut ends of $o$ and $b$. We have an \"ubercrossing projection of $K_1\#K_2$.  

\item Since strands $b$ and $o$ are one level apart in height within the crossing, we can remove the loop connecting $b$ and $o$ by pulling it out through the multi-crossing, as in Figure \ref{steps}(g). This eliminates two strands in the multi-crossing. We obtain an \"ubercrossing projection with $(\ddot{u}(K_1)+p(K_2)-1)-2$ strands passing through the multi-crossing, as in Figure \ref{steps}(h).
\end{enumerate}

Thus, $\ddot{u}(K_1\# K_2)\leq \ddot{u}(K_1)+p(K_2)-3$.

\end{proof}

\begin{remark} 
For every achiral knot, there exists a minimal \"ubercrossing projection with a left ribbon and a minimal \"ubercrossing projection with a right ribbon. This is because when we reflect a projection all right ribbons become left ribbons. The conditions on Theorem \ref{comp} are only relevant when all minimal \"ubercrossing projections of $K_1$ have only left ribbons or only right ribbons. We have not yet found any examples of such knots.
\end{remark}

\begin{cor} Let $K_1$ and $K_2$ be knots. Then, 
\[\ddot{u}(K_1\#K_2)\leq \min\{\ddot{u}(K_1)+p(K_2)-1,\ddot{u}(K_2)+p(K_1)-1\}.\] 
\end{cor}

\begin{proof} Without loss of generality, assume $\ddot{u}(K_1)+p(K_2)-1\leq \ddot{u}(K_2)+p(K_1)-1$. Assume $K_2$ has a right bottom ribbon. The case for a left bottom ribbon follows similarly. If $K_1$ does not have a left ribbon, then add a trivial petal that is a left ribbon at any height in the \"ubercrossing. We draw an \"ubercrossing projection of the disjoint union of $K_1$ and $K_2$ as done in Theorem \ref{comp}. Use the projection of the minimal pre-petal projection of $K_2$ and the \"ubercrossing projection of $K_1$ with the added left ribbon. There will be $p(K)-1+\ddot{u}(K) +2$ strands passing throught the single crossing. When we compose $K_1$ and $K_2$ by connecting the right bottom ribbon of $K_2$ and the left ribbon of $K_1$, as done in Theorem \ref{comp}, we eliminate 2 of the strands passing through the single crossing. We obtain $\ddot{u}(K_1\#K_2)\leq \ddot{u}(K_1)+p(K_2)-1$.

\end{proof}

\begin{cor}\"{u} $(T_{r_1, r_1+1} \# \cdots \# T_{r_n, r_n +1}) = (\sum_{i = 1}^n 2r_i) - 2(n-1)$.
\label{torus-comp}
\end{cor}

\begin{proof} It follows from the fact that $p(T_{r,r+1}) = 2r+1$ and $\ddot{u}((T_{r,r+1}) = 2r$ that for these knots, the pre-petal diagram realizes the  \"ubercrossing number. Thus, by Theorem \ref{comp},  $\ddot{u}(T_{i, i+1}\#T_{j, j+1})\leq 2i+2j-2$. We claim the composition of two pre-petal projections according to the method described in Theorem \ref{comp} will have both a left and right ribbon. The top strand of a pre-petal projection will always be part of one left ribbon and one right ribbon. When we compose the knots the bottom strand of the pre-petal projection that acts as $K_2$ is cut and glued. This local operation does not effect the types of ribbon the top strand of $K_2$ is a part of. Thus, the top strand of $K_2$ forms both a left and right ribbon in the compostion. This guarantees that we can apply Theorem \ref{comp} when composing additional $(r,r+1)$-torus knots. As we compose additional torus knots, we obtain $\ddot{u} (T_{r_1, r_1+1} \# \cdots \# T_{r_n, r_n +1}) \leq (\sum_{i = 1}^n 2r_i) - 2n + 2$.

By Lemma \ref{bridge}, we know that $\ddot{u}(K) \ge 2b(K)$. Since $b(K_1\# K_2) = b(K_1) + b(K_2) -1$ by \cite{Schubert}, we have:  
\begin{align*}
\ddot{u}(T_{r_1, r_1+1} \# \cdots \# T_{r_n, r_n +1}) &\ge 2b(T_{r_1, r_1+1} \# \cdots \# T_{r_n, r_n +1}) \\
& = 2(b(T_{r_1, r_1 +1}) + \cdots + b(T_{r_n, r_n +1}) - (n-1)) \\
& = 2((\sum_{i=1}^n r_i) - (n-1))
\end{align*}
Therefore, we have $\ddot{u}(T_{r_1, r_1+1} \# \cdots \# T_{r_n, r_n +1}) = (\sum_{i = 1}^n 2r_i) - 2(n-1)$
\end{proof}


\section{Unknotting Number and Petal Number}\label{unknotting}

In this section, we show how to unknot a petal projection by changing the relative heights of the strands in the single $n$-crossing. From this, we obtain an upper bound on unknotting number.

\begin{definition}
The \emph{unknotting number} of a knot $K$, denoted $u(K)$, is the minimum number of times that $K$ has to pass through itself in 3-space, or equivalently, the minimum number of crossing changes needed to change a planar projection of $K$ into the trivial knot, over all projections of $K$.
\end{definition}

As mentioned in the introduction, there is an equivalence of petal permutations representing a given knot that reflects the fact the top strand can be removed and then placed on the bottom. We define an equivalence class, $\overline{[\sigma]}$, on petal permutations so that all petal permutations that can be obtained from each other in this way are equivalent: $(1 ,a_2 ,\ldots,a_p) \equiv (1,b_2, \ldots b_p)$ if the permutation $([a_2-1]_p, [a_3-a_2]_p,\dots,[a_p-a_{p-1}]_p, [1-a_p]_p)$ can be obtained by cyclically rotating the permutation $([b_2-1]_p, [b_3-b_2]_p,\dots,[b_p-b_{p-1}]_p, [1-b_p]_p)$ where $[a]_p$ denotes $a$ mod $p$. The value $a_i-a_{i+1}$ represents the change in height as we travel around the $a_i-a_{i+1}$ petal. The equivalence relation makes it so that it does not matter which strand we label 1 as long as the sequence of height changes is the same up to rotation.

In the following proof, we use petal permutations to keep track of how we change the relative heights of strands and remove strands in the process of obtaining the unknot. We then use the equivalence classes of permutations to show that the bound is tight for $(r, r+1)$-torus knots. 

\begin{theorem}
Let $K$ be a knot, then 
$$u(K) \leq \frac{(p(K)-1)(p(K)-3)}{8}.$$
Equality holds if and only if $K=T_{r,r+1}$ or $K=\overline{T_{r,r+1}}$ for some integer $r$.
\end{theorem}

\begin{proof}
Let $p=p(K)$. Consider a petal projection of the knot $K$ with $p$ petals and a permutation $\sigma=(1,a_2,\ldots,a_p)$.

Let $a_i$ and $a_j$ be two values in $\sigma$. Define $d_p(a_i,a_j)=\min \{d \ge 0|[a_i+d]_p \equiv [a_j]_p \, \mathrm{or} \, [a_i-d]_p\equiv [a_j]_p\}$. Intuitively, this is a measure of how many petals on the knot we must pass through to travel from the strand at height $a_i$ to the strand at height $a_j$. Clearly, $d_p(a_i,a_j)\leq \frac{p-1}{2}$.

If $d_p(a_{i-1},a_i)=1$, then there is a petal that connects the strands with height $a_{i-1}$ and $a_i$. We call this type of petal a {\bf trivial petal} because we can remove it by pulling it through the multi-crossing. Doing so yields a petal projection with $p-2$ petals, as shown in Figure \ref{fig: trivial petal}.

\begin{figure}
\begin{center}
\includegraphics[width=.4\textwidth]{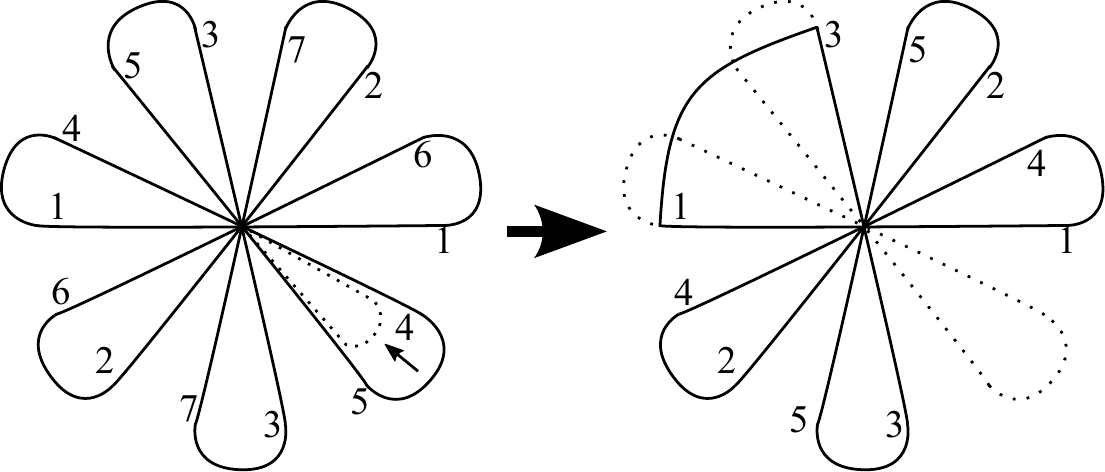}
\caption{Pulling out the trivial petal in $(1,4,5,3,7,2,6)$ leads to $(1,3,5,2,4)$}
\label{fig: trivial petal}
\end{center}
\end{figure}

We now present a method of unknotting a petal projection by changing the relative heights of the strands without changing the rest of the projection. We change the heights of the strands to obtain trivial petals, which we can then remove. We continue this process until we reach a projection with three petals because any projection with three petals is a trivial knot. 

Let $a_i$ and $a_{i-1}$ be numbers in a permutation of corresponding to a petal projection with $p$ petals. Suppose $d_p(a_{i-1},a_i)=d$. Without loss of generality, assume $[a_i]_p=[a_{i-1}+d]_p$. To obtain a trivial petal, we must move the strand initially at the $a_{i-1}$ level so that it is at the $a_i+1$ level. To do so, the strand must pass through the knot $d-1$ times. Next, we remove the trivial petal and obtain a $(p-2)$-petal projection. Since $d_p(a_i,a_j)\leq \frac{p-1}{2}$, the knot will pass through itself at most $(\frac{p-1}{2}-1)$ times in changing the $p$-petal projection into a $(p-2)$-petal projection. We repeat this process until we reach a projection with three petals. We obtain the following bound on unknotting number:

\begin{align*}
u(K)&\leq \bigg(\frac{p-1}{2}-1\bigg)+\left(\frac{(p-2)-1}{2}-1\right)+\dots+\left(\frac{3-1}{2}-1\right)\\
&=\sum_{i=0}^{\frac{p-3}{2}} i =\frac{(p-1)(p-3)}{8}.
\end{align*}



We know from \cite{Multi} that, $p(T_{r,r+1})= p(\overline{T_{r,r+1}})=2r+1$. By \cite{KM}, $u(T_{p,q})=\frac{(p-1)(q-1)}{2}$.  Thus, \[u(T_{r,r+1})=\frac{r(r-1)}{2}=\frac{(2r+1-1)(2r+1-3)}{8}=\frac{(p-1)(p-3)}{8}\] Therefore, $T_{r, r+1}$ torus knots realize the upper bound. 

Now assume that a knot $K$ realizes the upper bound. Then  $d_p(a_{i-1},a_i)=\frac{p-1}{2}$ for all $i \in \{2,\ldots,p\}$. Let $k=\frac{p-1}{2}$. This condition only holds for two equivalence classes of petal permutations,
\begin{align*}
 \overline{[\sigma_1]}=([1]_p,[1+k]_p,[1+2k]_p, \ldots)\\
\overline{[\sigma_2]}=([1]_p,[1-k]_p,[1-2k]_p, \ldots)
\end{align*}
 
Note that these equivalence classes each contain a single element. Therefore, there are only two knots that could possibly realize this bound (one with $[\sigma_1]$, and the other with $[\sigma_2]$). These are the only two permutations because if a petal permutation of $p$ strands is not in $\overline{[\sigma_1]}$ or in $\overline{[\sigma_2]},$ then there is a way to obtain a trivial petal by changing the height of a strand so that the knot passes through itself fewer than $k=\frac{p-1}{2}$ times, meaning the bound is not met. 

The $(r, r+1)$-torus knot has petal number $p=2(\frac{p-1}{2})+1=2r+1$. Since the $(r, r+1)-$torus knot is chiral, it can be denoted by at least two distinct permutations.  The $(r, r+1)-$torus knot (and its mirror) realize the unknotting bound, as proven above, therefore the $(r, r+1)-$torus knot and its mirror must have petal permutations that are equivalent to $\sigma_1$ and $\sigma_2$. Since a petal permutation can only represent one knot, only $(r,r+1)$-torus knots and their mirrors realize the bound.
\end{proof}

\begin{cor}\label{unique}
The unique minimal petal permutation of $T_{r,r+1}$ is $$(1,r+1,2r+1,r,2r,\ldots,2,r+2).$$
\end{cor}

\begin{proof}
We may unknot any knot in a petal projection by passing strands through one another to change their heights in the multi-crossing and then removing trivial petals as described above. In a permutation on $p$ strands, the maximum number of crossing changes needed to create a trivial petal is $\frac{p-3}{2}$. Consequently, any permutation on $2r+1$ strands other than 
$$(1,r+1,2r+1,r,2r, \ldots ,2,r+2)$$
\begin{center}
or
\end{center}
$$(1, r + 2, 2, r + 3, 3, \ldots, 2r + 1,r+1)$$
has a lower unknotting number than either $T_{r,r+1}$ or $\overline{T_{r,r+1}}$.
\end{proof}

Note that most knots have more than one minimal petal permutation, and therefore Corollary \ref{unique} is unusual.


\section{Crossing Number and Petal Number}\label{petal into star}

In this section, we present  an upper bound on crossing number in terms of petal number and show that the $(r, r+1)$-torus knots realize this bound. 

\begin{theorem} \label{big} Let $K$ be a knot with petal number $p$. Then, \[\label{petalcrossing}c(K)\leq \frac{p^2-2p-3}{4},\] where $c(K)$ denotes the $c_2$-crossing number of $K$.\label{pcbound}
\end{theorem}

\begin{proof} We will prove this theorem by showing how to isotope a petal projection of $K$ into a projection with $\frac{p^2-2p-3}{4}$ $c_2$-crossings.  Figure \ref{pic} illustrates this process for a knot satisfying $p(K)=9$. First, perturb the single $p$-crossing to form a star pattern with $\frac{p(p-1)}{2}$ double crossings. When resolving, make sure that each strand is pushed in the direction that makes its petals smaller. This isotopy creates monogons on the ends of each petal. Then perform $p$ Type I Reidmeister moves to remove the monogons on the ends of each arm of the star. This decreases the the number of crossings by $p$.  Moreover, this also reduces the number of intersections of each strand from $(p-1)$ to $(p-3)$.

Next, take the strand that was on top in the petal projection and isotope it so that it extends outside the star. There are two ways to surround the star, but we perform this isotopy so that the top strand surrounds the fewest number of petals. This operation, which we call strand removal, eliminates $p-3$ crossings from within the star region. Remove the second highest strand in the same way. This operation eliminates at least $p-4$ crossings from the star region, but creates a new crossing with the previous strand we removed. The $i^{th}$ iteration of strand removal eliminates $p-2-i$ crossings from the star region and creates $i-1$ crossings with strands that have already been removed. Thus, the number of crossings of the projection changes by $i-1-(p-2-i)=2i+1-p$. We iterate the strand removal $\frac{p-3}{2}$ times to ensure that each iteration decreases the number of crossings, meaning $2i+1-p$ is negative. We obtain projection, $P$ with crossing number as follows, 
\begin{align*}
c(P)&=\frac{p(p-1)}{2}-p+\sum_{i=1}^{ \frac{p-3}{2}} 2i+1-p\\
&=\frac{p(p-1)}{2}-p+\frac{(2+1-p)+(p-3+1-p)}{2}\bigg(\frac{p-3}{2}\bigg)\\
&=\frac{2p(p-1)-4p+(1-p)(p-3)}{4}\\
&=\frac{p^2-2p-3}{4}.
\end{align*}
Therefore, $c(K)\leq\frac{p^2-2p-3}{4}$, as desired.
\end{proof}

\begin{figure}[h!]
\centering
\includegraphics[scale=.7]{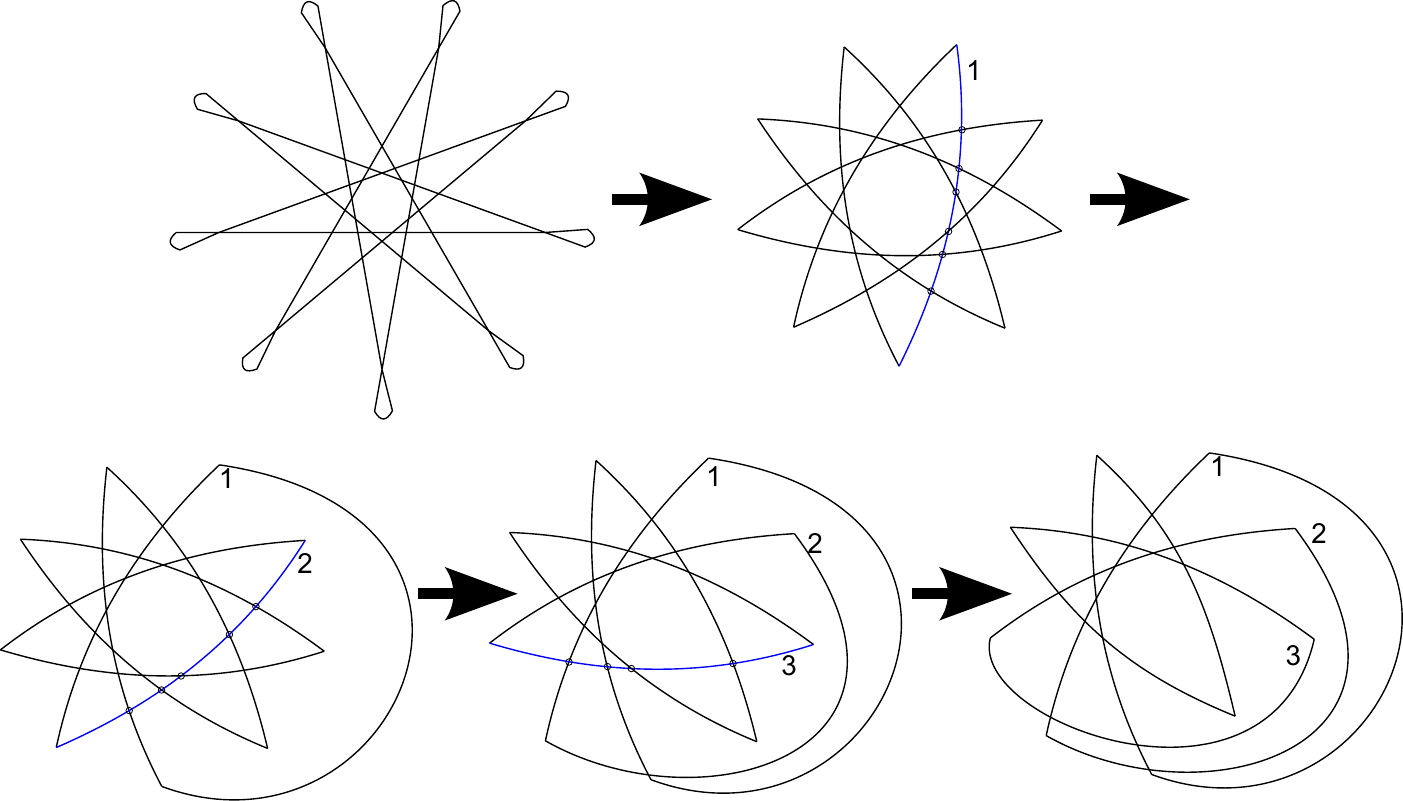}
\caption{Step 1: Perturb the single crossing to form a star pattern.  Step 2: Remove monogons to decrease the number of crossings by $p=9$. Step 3: The first iteration of strand removal decreases the number of crossings by $p-3=6$. Step 4: The second iteration of strand removal decreases the number of crossings by $p-4-1=4$. Step 5: The third iteration of strand removal decreases the number of crossings by $p-5-2=2$.  This is the final iteration because removing another strand does not guarantee that we decrease the number of crossings.}
\label{pic}
\end{figure}

\begin{cor} For all torus knots of the form $T_{r,r+1}$,  
\[c(K)=\frac{p^2-2p-3}{4}.\]
\end{cor}
\begin{proof}
For a torus knot of the form $T_{r,r+1}$, $c(T_{r,r+1})=r^2-1$  by \cite{KM}, and $p(T_{r,r+1})=2r+1$ by \cite{Multi}. So, in fact, this family of knots realizes the above equality.
\end{proof}

 We now have an infinite family of knots for which the bound in Theorem \ref{big} is realized.


\section{Results on Petal Algorithm}\label{algorithm}

In this section we investigate the petal algorithm, described in \cite{Multi}, which takes a double-crossing projection and generates a petal projection. We consider the construction of minimal petal projections via the petal algorithm. The first question on might ask is whether there is always a projection of a knot such that when the petal algorithm is applied to it, one obtains the petal number of that knot.

\begin{theorem}\label{algorithm-thm}
There exists a double crossing projection of every knot $K$ such that applying the petal algorithm generates the minimal petal projection of $K$.
\end{theorem}

\begin{proof}
 We take a minimal petal projection and reverse the petal algorithm in order to obtain a double crossing projection. We are essentially showing that any petal projection can be viewed ``sideways" so that there are only double crossings. The reversal of the petal projection is described below.
 
\begin{enumerate}
\item Given any minimal petal projection, take the top strand off to obtain a pre-petal projection, and re-number the strands according to their new heights.  Rotate the pre-petal projection until the rightmost point of the projection is located on the nesting loop.
\item If the pre-petal projection has $n$ strands, draw $n$ horizontal lines that intersect an imaginary vertical line we  denote by $A$. Number each strand according to its height, where the top strand is numbered one.
\item On the right side of $A$, connect the horizontal strands that correspond to the strands in the nesting loop with a vertical strand so that it is the rightmost strand. Connect the remaining horizontal lines with vertical lines according to the pre-petal projection permutation. Note that we connect the horizontal lines so that half of the pre-petal projection is represented on the right of $A$ and the other half is represented on the left of $A$. 
\item Chose an orientation for the knot such that the starting point is any point on the rightmost strand and travel along the knot. When we encounter a crossing, we determine whether a crossing is over or under in the following manner:
\begin{itemize}
\item If the crossing is on the right of $A$, and has not been traveled through, then make the crossing an over crossing according to the orientation.
\item If the crossing is on the left of $A$, and has not been traveled through, then make the crossing an under crossing according to the orientation.
\end{itemize}
\end{enumerate}
Note that we make the crossings over and under in the manner explained in Step $4$ because that is how it is isotoped when doing the petal algorithm from a double crossing to a petal projection.
We now have a double crossing projection that can generate a minimal petal projection if we apply the petal algorithm to it.
\end{proof}

Although Theorem~\ref{algorithm-thm} guarantees that a minimal petal projection of a knot can be obtained by applying the petal algorithm to some projection of the knot, the projection that generates the minimal petal projection of a knot need not be a minimal crossing projection, as shown in this next theorem.

\begin{thm}
\label{2-braid thm}
Given any 2-braid knot $B$ with crossingt number $c > 3$, a minimal petal projection of B cannot be obtained by applying the petal algorithm to a minimal crossing projection of $B$.
\end{thm}

\begin{proof}
Given a minimal double-crossing projection of $B$ with some orientation, which is uniquely determined on the projection sphere by \cite{Kau2}, \cite{Muras}, \cite{Th} and \cite{MT}, apply the petal algorithm. Note that as we travel along the projection, starting at any point on the projection preceding an overcrossing, we will pass through all crossings once before passing through any crossing a second time. Since $B$ is alternating, as we travel along the knot there will be $\frac{c +1}{ 2} $ overcrossings and $\frac{c -1}{ 2} $ undercrossings before we reach a crossing a second time. The mirror of $B$ has $\frac{c +1}{ 2} $ undercrossings and $\frac{c-1}{2}$ overcrossings, however since it has the same petal number, for convenience we will use $B$. Recall that in the petal algorithm, overcrossings are moved to the right of $A$ and undercrossings are moved to the left of $A$, where $A$ is the vertical axis used in the petal algorithm.

Of the $c$ bigons in the projection, all but one will have one crossing labelled $U$ and one crossing labelled $O$. Hence the axis $A$ must cross both of the edges of each of these bigons. Hence,  the knot crosses $A$ at least $2(c -2)$ times. The resulting pre=petal diagram will have $2c -2$ loops and thus, the resulting petal projection will have at least $2c - 1$ petals.  

We know by \cite{Multi}, that $p(B)= c +2$, therefore the minimal double crossing projection of K does not yield a minimal petal projection using the petal algorithm.
\end{proof}

\begin{remark}
Figure \ref{2-braidpetal} shows the double-crossing projection of the 2-braid knot with crossing number $c$ that generates the minimal petal projection of the corresponding knot.  This projection has $\frac{3c-3}{2}$ crossings and  it packs all of these crossings onto three strands. The strand labelled 1 has $\frac{c-3}{2}$ crossings on it, the strand labelled 2 has $\frac{c-1}{2}$ crossings on it, and the strand labelled 3 has $\frac{c+1}{2}$ crossings on it. 
\end{remark}

\begin{figure}[h!]
\centering
\includegraphics[scale=.4]{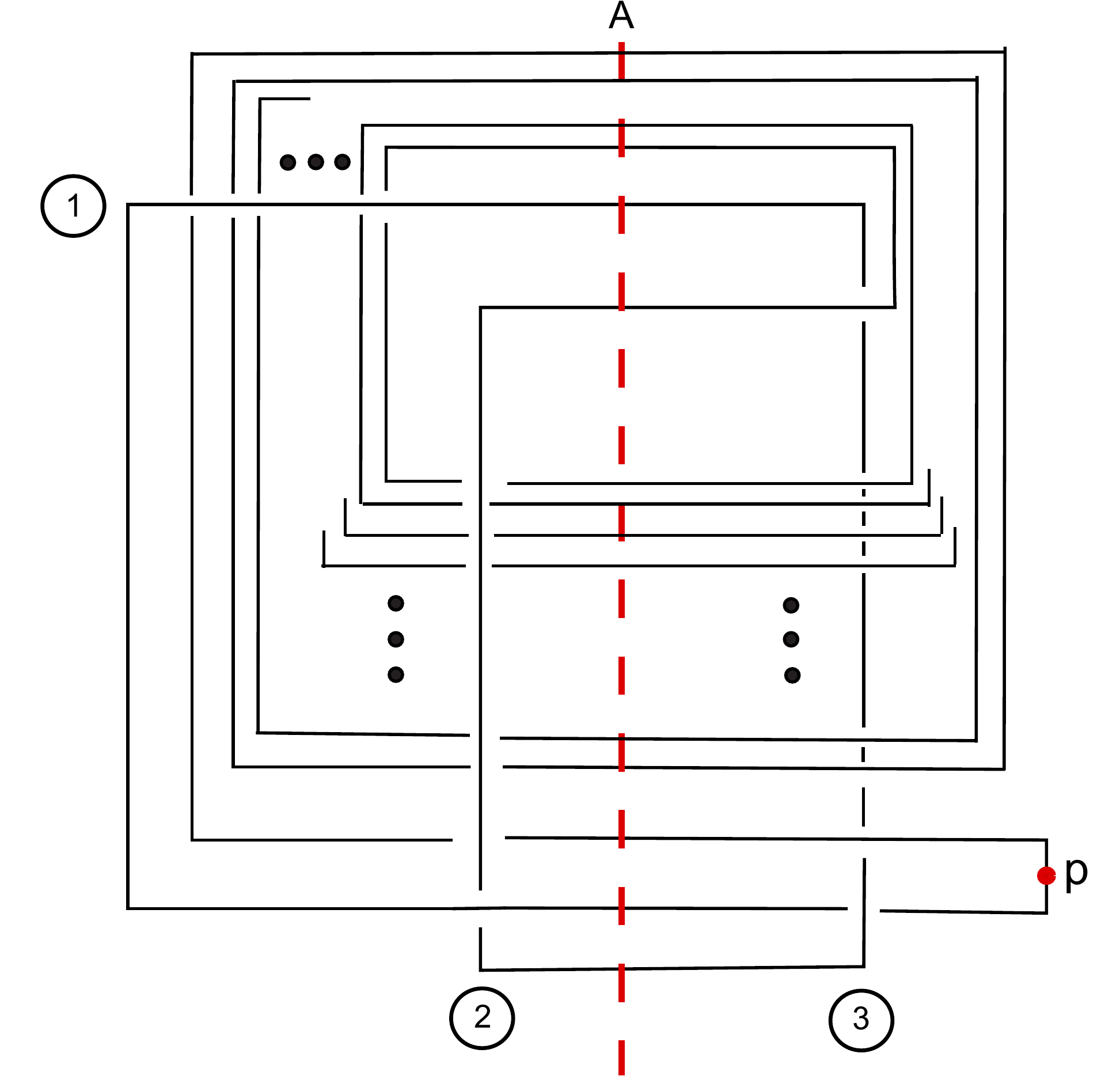}
\caption{The double-crossing projection of 2-braid knots that generates the corresponding minimal petal projections when  the petal algorithm is applied.}
\label{2-braidpetal}
\end{figure}

\pagebreak


\begin{thebibliography}{9}

\bibitem{Triple} C. Adams, \emph{Triple crossing number of knots and links}, C. Adams, J. Knot Theory Ram., \textbf{22} (2013),1350006-1--1350006-17.

\bibitem{Quadruple} C. Adams, \emph{Quadruple crossing number of knots and links}, ArXiv 1211.2726 (2012), to appear in Math. Proc. of the Camb. Phil Soc.

\bibitem{Multi}
C. Adams, T. Crawford, B. DeMeo, M. Landry, A. Tong Lin, M. Montee, S. Park, S. Venkatesh, and F.Yhee, \emph{Knot projections with a single multi-crossing}, ArXiv 1208.5742 (2012).

\bibitem{Adams1}
C. Adams, O.Capovilla-Searle,
J. Freeman, D. Irvine, S. Petti, D.Vitek, A.Weber and 
S. Zhang, \emph{Multicrossing numbers for knots and the span of the bracket polynomial }, preprint.

\bibitem{Kau2}
L. Kauffman, \emph{State models and the Jones polynomial}, Topology
 \textbf{26} (1987), pp. 394--407.
 
 \bibitem{KM} P. Kronheimer, T. Mrowka,\emph{Gauge theory for embedded surfaces I}, Topology \textbf{32} (1993) 773--826.
 
  \bibitem{MT}
W. Menasco, M. Thistlethwaite, \emph{The classification of alternating links}
 Ann. of Math. (2) 138 (1993), no. 1, 113--171.
 
 \bibitem{Muras}
K. Murasugi, \emph{Jones polynomials of alternating links},
Trans. A.M.S.,  \textbf{295} (1986), pp. 147--174.


\bibitem{Schubert}
H. Schubert, \emph{Uber eine numerische knoteninvariante},
Math. Z. \textbf{61} (1954), pp. 245--288.
	
	
\bibitem{Th}
M. Thistlethwaite, \emph{A spanning tree expansion of the Jones polynomial},
Topology,  \textbf{26} (1987), pp. 297--309.

	

	



	
\end{thebibliography}
\end{document}